\documentclass[12pt]{amsart}
\usepackage{amsmath,amssymb,graphicx,mathrsfs}
\usepackage{fullpage}
\usepackage{hyperref}
\usepackage{tikz}
\usepackage{color}
\def\al{\alpha}
\def\nn{\nonumber}

\def\nn{\nonumber}

\def\Z{{\mathbb Z}}
\def\C{{\mathbb C}}

\def\R{\mathbb R}

\DeclareMathOperator{\diag}{diag}
\DeclareMathOperator{\Ad}{Ad}

\DeclareMathOperator{\Hom}{Hom}

 \numberwithin{equation}{section}
  \newtheorem{corollary}[equation]{Corollary}
    \newtheorem{lemma}[equation]{\bf Lemma}
    \newtheorem{conjecture}[equation]{\bf Conjecture}
        \newtheorem{examplent}[equation]{\bf Example}
    \newtheorem{propositionnt}[equation]{\bf Proposition}
        \newtheorem{remark}[equation]{\bf Remark}
        \makeatletter
\g@addto@macro\th@remark{\thm@headpunct{}}
\makeatother
  \theoremstyle{remark}
      \newtheorem{proposition}[equation]{\bf Proposition}
    \newtheorem{example}[equation]{\bf Example}
\newtheorem{definition}[equation]{\bf Definition}
       \newtheorem{theorem}[equation]{\bf Theorem}

\newcommand{\ts}{\textstyle}

\DeclareMathOperator{\GL}{GL}
\DeclareMathOperator{\SL}{SL}
\DeclareMathOperator{\re}{Re}

\title{\boldmath The functional equations of Langlands Eisenstein series for ${\rm \bf SL}(n,\mathbb Z)$}
\author{Dorian Goldfeld \and Eric Stade \and Michael Woodbury}
\address{Department of Mathematics\\Columbia University\\ 2990 Broadway \\ New York, NY 10027, USA}
\email{goldfeld@columbia.edu} 
\address{Department of Mathematics\\
University of Colorado Boulder\\
Boulder, Colorado 80309, USA}
\email{stade@colorado.edu}
\address{Department of Mathematics \\Rutgers, The State University of New Jersey\\
110 Frelinghuysen Rd\\
Piscataway, NJ 08854-8019, USA}
\email{michael.woodbury@rutgers.edu} 

\thanks{Dorian Goldfeld is partially supported by Simons Collaboration Grant Number 567168.}

\begin{document}

\begin{abstract} This paper presents a very simple explicit description of Langlands Eisenstein series for $\SL(n,\mathbb Z)$. The functional equations of these Eisenstein series are heuristically derived from the functional equations of certain divisor sums and certain Whittaker functions that appear in the Fourier coefficients of the Eisenstein series.  We conjecture that the functional equations are unique up to a real affine transformation of the $s$ variables defining the Eisenstein series and prove the uniqueness conjecture in certain cases.\end{abstract}

\maketitle

\tableofcontents

\centerline{\it Dedicated to the memory of Chen Jingrun
}
\section{ \large \bf Introduction} 

\subsection{Early history of the analytic theory of Eisenstein series}

As remarked by Moeglin and Waldspurger \cite{s11}, the analytic theory of Eisenstein series really began with the work of
 Maass \cite{s10} who formally defined the series
 \begin{equation}\label{Eis2}E(z,s) := \frac12\underset{c,d\in\mathbb Z}{\sum_{(c,d)=1}} \,\frac{y^s}{|cz+d|^{2s}}\end{equation}
 which converges absolutely for $s\in\mathbb C$ and $\Re(s)>1$ for all $z = x+iy$  in the upper half-plane, i.e., $x\in\mathbb R, y>0.$ Let $\zeta(s)$ denote the Riemann zeta function which satisfies the functional equation
 \begin{equation}\label{zetaFE}
 \zeta^*(s) := \pi^{-\frac{s}{2}} \Gamma\left(\frac{s}{2}\right)\zeta(s) = \zeta^*(1-s).
 \end{equation}
Maass obtained the meromorphic continuation and functional equation of the completed Eisenstein series
$$E^*(z,s) := \zeta^*(2s) E(z,s) = E^*(z,1-s)$$
from the Fourier expansion of $E(z,s)$ together with the functional equation (\ref{zetaFE}). This approach was generalized by Roelcke \cite{s12} to discrete groups commensurable with $\SL(2,\mathbb Z)$ and completed by Selberg \cite{s13}, \cite{MR0176097} for all Eisenstein series on $\GL(2,\mathbb R).$ In his talk at the International Congress in Stockholm 1962, Selberg \cite{MR0176097} presented a new proof of the functional equation of rank one Eisenstein series  which did not make use of the Fourier expansion of   these series
except for the constant term. This approach was generalized by Langlands, (see \cite{s8}, \cite{s9}, \cite{s11}) who defined more general Eisenstein series in the higher rank case and extended Selberg's proof of the meromorphic continuation and functional equations of Eisenstein series. The basic principle in Selberg's proof is to show the analytic continuation of the Eisenstein series and its constant term simultaneously by using the fact that the resolvent of an operator has analytic continuation to the complement of its spectrum. 
In 1967 Selberg found another proof of the functional equation of Eisenstein series which was not published but shown to Dennis Hejhal, Paul Cohen, and Peter Sarnak which Selberg suggested would also work in the case of higher rank, but it took at least two decades before Selberg's claim was vindicated. In the 1980's Joseph Bernstein simplified Selberg's second proof. More recently Bernstein and Lapid \cite{s1} found a {\it ``soft''} uniform proof of the meromorphic continuation and functional equations of Eisenstein series induced from a general automorphic form (not necessarily cuspidal or in the discrete spectrum).  

\subsection{Elementary introduction to Langlands Eisenstein series for \boldmath $\SL(n,\Z)$} We now present  a very elementary explanation of the notation for Langlands Eisenstein series. For a formal definition see \S \ref{EisSeriesDef}.

 Let $n\ge2$.  The Langlands Eisenstein series for $\SL(n,\Z)$ depends on an integer partition $$n=n_1+n_2+\cdots+ n_r, \qquad \big( n_1,n_2,\ldots, n_r\in \Z_{>0} \;\, \text{\rm and}\;\,  2\le r\le n \big),$$  which we denote by $\mathcal P=\mathcal P_{n_1,n_2,\ldots n_r}$. In addition, the Langlands Eisenstein series for $\SL(n,\Z)$  also depends on a tensor product of automorphic functions 
 $$\Phi:= \phi_1\otimes \phi_2\otimes \cdots\otimes \phi_r,$$
  where each $\phi_j:\GL(n_j,\mathbb R)\mapsto \C$ is a smooth function invariant under left multiplication by the discrete subgroup $\SL(n_j,\Z)$ and the center of $\GL(n_j,\mathbb R)$, and right invariant by ${\rm O}(n_j,\mathbb R).$  We denote the Langlands Eisenstein series associated to $\mathcal P$ and $\Phi$ by $E_{ \mathcal P,\Phi}(g,s)$, where $s=(s_1,s_2,\ldots, s_r)\in \C^r$ satisfies $\sum\limits_{i=1}^r n_i s_i=0$.

\subsection{Motivation and Main Theorems of this paper}
After seeing the paper \cite{s3} (where a template method for computing the first coefficient of Langlands Eisenstein series for Chevalley groups is presented), Peter Sarnak raised the question if it might be possible to prove the meromorphic continuation and functional equation of Langlands Eisenstein series for $\SL(n,\mathbb Z)$ by the original method of Maass which just uses the explicit form of each of the Fourier coefficients of Eisenstein series. 

 The  aim of this paper is to show that the non-constant Fourier coefficients of Langlands Eisenstein series for $\SL(n,\Z)$ all satisfy the same functional equations and are entire functions of the complex variables  defining the Eisenstein series. It was proved  in \cite{s6} that the first coefficient of $E_{\mathcal P,\Phi}$ is given as a reciprocal of a certain product of completed Rankin-Selberg L-functions. If we multiply $E_{\mathcal P,\Phi}$ by this product of completed Rankin-Selberg L-functions we obtain the normalized Eisenstein series $E^*_{\mathcal P,\Phi}$   defined in our main Theorem \ref{maintheorem}. We will confirm by direct computation of the Fourier coefficients (see    Theorem \ref{EisFunctionalEquations}) that the functional equation  of  $E^*_{\mathcal P,\Phi}(g,s)$ proved by Langlands (see \cite{s9}, \cite{s11}) is  given by
\begin{equation} \label{FEforE*}
\boxed{E^*_{\mathcal P,\Phi}(g,s)=E^*_{\sigma\mathcal P,\sigma\Phi}(g,\sigma s)}
\end{equation}
where the permutation $\sigma\in S_r$ satisfies
\begin{align*} \sigma \mathcal  P= \mathcal P_{n_{\sigma(1)},n_{\sigma(2)},\ldots,n_{\sigma(r)}},\quad
  \sigma \Phi= \phi_{\sigma(1)}\otimes\phi_{\sigma(2)}\otimes \cdots \phi_{\sigma(r)},\quad
\sigma s = \big(s_{\sigma(1)},s_{\sigma(2)},\ldots,s_{\sigma(r)}\big).\end{align*}
 For some simple concrete examples of functional equations of Langlands Eisenstein series see \S\ref{SmallRankExamples}.
 
 \vskip 8pt
 We also conjecture that the functional equations \eqref{FEforE*} are unique
 in the sense that if there exists a real valued affine  transformation $\mu(s)$ of the variables $s=(s_1,s_2,\ldots,s_k)$ then
  $\mu$ has to be a permutation in $S_r$. See \S \ref{UniquenessSection} where this conjecture is stated and proved in the case where $\Phi = \phi_1\otimes \phi_2\otimes\cdots\otimes\phi_r$ and $\phi_1=\phi_2=\cdots =\phi_r.$
 As far as we know the uniqueness of functional equations of Langlands Eisenstein series has not been investigated before.
 \vskip 8pt
  If we knew that every Fourier-Whittaker coefficient of the Langlands Eisenstein series $E^*_{ \mathcal P,\Phi}(g,s)$ had meromorphic continuation in all its complex variables and satisfied the same functional equations, then this would give a new significantly  simpler proof of the meromorphic continuation and functional equations of all Langlands Eisenstein series for $\SL(n,\Z)$.  We conjecture that it's enough to prove this for the constant Fourier-Whittaker coefficient and the $$(1,1,\ldots,1,\underbrace{p}_{j^{\rm th} {\tiny\rm\ entry}},1,\ldots,1)$$coefficients, for every prime $p$ and all $j$ with $1\le j\le n-1$.

\section{\large   \bf Basic notation}

\begin{definition} {\bf (Generalized upper half plane $\mathfrak h^n$).}
  We define the \emph{generalized upper half plane} as 
 \[ \mathfrak h^n := \GL(n,\mathbb R)/\left(\text{O}(n,\R)\cdot\R^\times   \right). \]
 By  the Iwasawa decomposition of $\GL(n)$ (see \cite[\S 1.2]{s5}) every element of $\mathfrak h^n$ has a coset representative of the form $g=x y$ where
\begin{equation}\label{eq:ymatrix-def}
x = \left(\begin{smallmatrix} 
  1 & x_{1,2} & x_{1,3}& \cdots  & & x_{1,n}\\
  & 1& x_{2,3} &\cdots & & x_{2,n}\\
  & &\hskip 2pt \ddots & & & \vdots\\
  & && & 1& x_{n-1,n}\\
  & & & & &1\end{smallmatrix}\right) \in {U}_n(\mathbb R), \qquad\;
 y =
    \left(\begin{smallmatrix} y_1y_2\cdots
    y_{n-1} & & & \\
    & \hskip -30pt y_1y_2\cdots y_{n-2} & & \\
    & \ddots &  & \\
    & & \hskip -5pt y_1 &\\
    & & &  1\end{smallmatrix}\right),
\end{equation}
with $y_i > 0$ for each $1 \le i \le n-1$.  The group $\GL(n,\R)$ acts as a group of transformations on $\mathfrak h^n$ by left multiplication.
\end{definition}

\begin{remark}\rm  In the case $n=2$, the above definition gives us the classical upper half-plane$$\mathfrak h^2=\biggl\lbrace \biggl(\begin{matrix}   1 &x\\
0   & 1\end{matrix}\biggr)\biggl(\begin{matrix}  y &0\\
0   & 1\end{matrix}\biggr)\bigg \vert\  x\in \R,\ y>0\biggr\rbrace.$$\end{remark}

Note that the summand $y^s/|cz+d|^{2s}$ figuring in the Eisenstein series \eqref{Eis2} is of the form  $f( \gamma z)$, where $f$ is the power function on $\mathfrak h^2$ given by $f(x+iy)=y^s$.  It is natural to define an analogous power function on $\mathfrak h^n$.

Consider a partition $n=n_1+n_2+\cdots+ n_r$, where $n_i\in \Z_{>0}$  for $1\le i\le r$.  We can then define a power function on matrices   
$\mathfrak m=  \left(\begin{smallmatrix} \mathfrak m_1 & * & \cdots &*\\
 & \mathfrak m_2 & \cdots & *\\
 &  & \ddots & \vdots \\
 & & & \mathfrak m_r\end{smallmatrix}\right) \in \GL(n,\mathbb R)$ and $r$-tuples $s=(s_1,s_2,\ldots, s_r)\in \C^r$ satisfying $\sum\limits_{i=1}^r n_i s_i=0$ as follows:
\begin{equation}\big |\mathfrak m\big |^s_{\mathcal P} \; := \; \prod_{i=1}^r\big |\det \mathfrak m_i\big|^{s_i},\label{eq:powerfunction}\end{equation}where $\mathcal P = \mathcal P_{n_1,n_2,\ldots n_r}$ denotes the partition.

\begin{remark}\rm The condition $\sum\limits_{i=1}^r n_i s_i=0$ assures that the above power function is invariant under multiplication by elements of the center of $\GL(n,\R)$.\end{remark}

 \begin{definition} {\bf (Langlands parameter).} A {\it Langlands parameter} for $\GL(n)$ is an $n$-tuple $\al=(\al_1,\al_2,\ldots,\al_n)\in \C^n$ satisfying $\sum\limits_{i=1}^n \al_i=0$.\end{definition}

   \begin{definition} {\bf (Langlands parameter for an automorphic form).} \label{AutLang} \rm Let $F\colon\mathfrak h^n\to \C$ be a smooth function invariant under the left action of $\SL(n,\Z)$ on $\mathfrak h^n$, and suppose that $F$ is an eigenfunction of all the $\GL(n,\R)$-invariant differential operators on $\mathfrak h^n$ (see \cite[\S 2.4]{s5}). We say that \emph{$F$ has Langlands parameter $\al=(\al_1,\al_2,\ldots \al_n)$} if $F$ has the same eigenvalues as the power function $$\big | * \big|^{\al+\rho_{_{\mathcal B}}}_{ \mathcal B},$$where $\mathcal B$ denotes the partition $n=1+1+\cdots+1$, and\begin{equation}\rho_{_{\mathcal B}}=\biggl(\frac{n-1}{2},\frac{n-3}{2},\ldots, \frac{1-n}{2}\biggr).\label{PsubB}\end{equation} \end{definition}

 \begin{remark}\rm The notation $\mathcal B$ is used to connote the Borel parabolic subgroup.  There is a one-to-one correspondence between partitions and parabolic subgroups, up to conjugacy, cf. Definition \ref{GLnParabolic}.\end{remark}
 
 \begin{examplent}\rm A Maass form for $\SL(2,\Z)$ with Laplace eigenvalue $\tfrac{1}{4}-\beta^2$ has Langlands parameter $(\al_1,\al_2)=(\beta,-\beta)$.\end{examplent}

 \section{ \large \bf Examples of Langlands Eisenstein series of small rank}
 \label{SmallRankExamples}
 
 The Eisenstein series \eqref{Eis2} is constructed by summing the translates 
 $({\rm Im}\,\gamma z)^s$ over all $\gamma\in \biggl(\begin{matrix}   1 &*\\
0   & 1\end{matrix}\biggr)\backslash\SL(2,\Z)$.    That Eisenstein series is essentially the same as the following example.

\begin{example}{\bf (\boldmath $\SL(2,\Z)$ Eisenstein series).} \label{sl2Borel}\rm Here we have the partition $2=1+1$.  

Let $s=(s_1,s_2)\in \C^2$ with $s_1+s_2=0$.  Then the power function \eqref{eq:powerfunction} is
$$\bigg|\bigg(\begin{matrix}   1 &x\\
0   & 1\end{matrix}\bigg)\bigg(\begin{matrix}  y &0\\
0   & 1\end{matrix}\bigg) \bigg|^s_{\mathcal P_{1,1}}=y^{s_1}$$and we can define the Eisenstein series
$$E_{\mathcal P_{1,1}} (g,s)= \hskip 6pt \frac12 \cdot \hskip -6pt \sum_{\gamma\in (\begin{smallmatrix}   1 &*\\
0   & 1\end{smallmatrix})\backslash \SL(2,\Z)} \big|\gamma g\big|^{s+(1/2,-1/2)}_{\mathcal P_{1,1}}
= \hskip 6pt \frac12 \cdot \sum_{\gamma\in (\begin{smallmatrix}   1 &*\\
0   & 1\end{smallmatrix})\backslash \SL(2,\Z)} ({\rm Im}\,\gamma z)^{s_1+1/2},$$where $g$ has Iwasawa form $\big(\begin{smallmatrix}   1 &x\\
0   & 1\end{smallmatrix}\big)\big(\begin{smallmatrix}  y &0\\
0   & 1\end{smallmatrix}\big)$ and $z=x+iy$.

\begin{remark}\rm Shifting $s$ by $\tfrac{1}{2}$ in the power function simplifies the functional equation, whose derivation we now explain.\end{remark}

The Fourier expansion
   \begin{align} \label{FourierExpEisensteinGL(2)}
 &E_{\mathcal P_{1,1}}(g, s ) \\&=  { {y^{ s_1 +\frac12} + \phi( s_1 +\ts{\frac12}) y^{\frac12- s_1 }}} \; + {{\frac{1}{ \zeta^*(2 s_1 +1)}}}
\sum_{m\ne0} \;\, {{ \sigma_{2s_1 }(m)  |m|^{-s_1}}}  \sqrt{y}\,K_{s_1}(2\pi \lvert m\rvert y) e^{2\pi imx}, \nonumber
\end{align}
where
{\small$$\zeta^*(s)=\pi^{-s/2}\,\Gamma(\tfrac{s}{2})\,\zeta(s),\quad \phi(s) = \frac{\zeta^*(2s-1)}{\zeta^*(2s)}, \quad\sigma_s(n) = \sum_{\substack{d|n\\ d>0}} d^s,  $$}and $K$ denotes the classical $K$-Bessel function, is well-known.

To see the functional equation of $E_{\mathcal P_{1,1}}(g,s)$ from the Fourier expansion, it is necessary to define$$E_{\mathcal P_{1,1}}^*(g,s)= \zeta^*(2 s +1) E_{\mathcal P_{1,1}}(g,s);$$that is, we are clearing the denominator.\end{example}

The main object of this paper is to show that the functional equations of all Langlands Eisenstein series for $\SL(n,\Z)$  can be easily seen by observing the Fourier coefficients in the Fourier expansion of the Eisenstein series.

In the case of $\SL(2,\Z)$, $$\int_0^1 E_{\mathcal P_{1,1}}^*\biggl(\Bigl(\begin{matrix}   1 &u\\
0   & 1\end{matrix}\Bigr)g,s\biggr)\, e^{2\pi i m u}\,du= {{ {{ \sigma_{2s_1 }(m)  |m|^{-s_1}}}  }}
 \sqrt{y}\,K_{s_1}(2\pi \lvert m\rvert y)  $$for $m\ne0$.  This immediately implies that, if $E_{\mathcal P_{1,1}}(g,s)$ satisfies a functional equation in $s$,  then each of its Fourier coefficients must also satisfy the same functional equation.   Since  the $m^{\rm th}$ Fourier coefficient, for $m\ne0$, is easily seen to be invariant under $s\to -s$, and  the constant term  $$ \zeta^*(2 s_1 +1)\Big({ {y^{ s_1 +\frac12} + \phi( s_1 +\ts{\frac12}) y^{\frac12- s_1 }}}\Big)$$ is readily seen (by the functional equation of the Riemann zeta function) to satisfy the same functional equation, it follows that the  Eisenstein series satisfies  $E_{\mathcal P_{1,1}}^*(g,s)=E_{\mathcal P_{1,1}}^*(g,-s)$.

\begin{example}{\bf (\boldmath The Eisenstein series $E_{\mathcal P_{1,1,1}}(g,s)$ for  $\SL(3,\Z)$).} \label{sl3Borel}
In the case $n=3$, the above definition of $\mathfrak h^n$ yields$$\mathfrak h^3=\left. \left\lbrace  xy=\left(\begin{matrix}   1 &x_1&x_3\\
0   & 1&x_2\\0&0&1\end{matrix}\right)\left(\begin{matrix}  y_1 y_2&0 &0\\0&y_1&0\\
0  &0 & 1\end{matrix}\right)\right\vert\  x_1,x_2,x_3\in \R,\ y_1,y_2>0\right\rbrace.$$
Let $s=(s_1,s_2,s_3)\in\C^3$ with $s_1+s_2+s_3=0$.  Then the power function is given by $$\big| d xyk|^s_{\mathcal P_{1,1,1}}:=\left|  \left(\begin{matrix}  y_1 y_2&0 &0\\0&y_1&0\\
0  &0 & 1\end{matrix}\right)\right|^s_{\mathcal P_{1,1,1}}=(y_1 y_2)^{s_1} y_1^{s_2}, $$where $ d$ is in the center of $\GL(3,\R)$ and $ k\in {\rm O}(n,\R)$.
    Then for $g\in \GL(3,\R)$ we have$$E_{\mathcal P_{1,1,1}}(g,s)=\sum_{\gamma\in \left(\begin{smallmatrix}   1 &* &*\\& 1 &*\\&
 & 1\end{smallmatrix}\right)\big\backslash \SL(3,\Z)} \big|\gamma g\big|^{s+(1,0,-1) }_{\mathcal P_{1,1,1}}.$$The shift by $(1,0,-1)$ makes the form of the functional equations as simple as possible.  Note that this shift corresponds to the case $n=3$ of the function $\rho_{_{\mathcal B}}$ of \eqref{PsubB}, and to the case $n=3$ and $\mathcal P=\mathcal B$ of the more general  $\rho$-function of Definition \ref{rhofunction} below.
 
    \begin{propositionnt} Let $g\in \GL(3,\R)$ and  $s=(s_1,s_2,s_3)\in\C^3$ with $s_1+s_2+s_3=0$.  Define $$E^*_{\mathcal P_{1,1,1}}(g,s)=\left(\prod_{1\le j<\ell\le 3}\zeta^*(1+s_j-s_\ell)\right)E_{\mathcal P_{1,1,1}}(g,s).$$
    Then  $E^*_{\mathcal P_{1,1,1}}(g,s)$ satisfies the functional equation

   $$E^*_{\mathcal P_{1,1,1}}(g,s_1,s_2,s_3)=E^*_{\mathcal P_{1,1,1}}(g,s_{\sigma(1)},s_{\sigma(2)},s_{\sigma(3)})$$for any $\sigma\in S_3$.  Moreover, this functional equation is unique in that, if $\mu$ is any real affine transformation of $s$ such that $$E^*_{\mathcal P_{1,1,1}}(g,s_1,s_2,s_3)=E^*_{\mathcal P_{1,1,1}}(g,\mu(s)),$$then $\mu(s)$ is a permutation of $s_1,s_2, s_3$.
   \end{propositionnt}

 \begin{proof}
It's well-known that $E_{\mathcal P_{1,1,1}}(g,s)$ has analytic continuation and satisfies six functional equations, including the trivial relation (see \cite[Ch.~7]{s2}).  
 As we did for $\SL(2,\Z)$, we will determine these functional equations heuristically by considering  just the $(m,1)^{\rm th}$ Fourier coefficient of $E_{\mathcal P_{1,1,1}}(g,s)$, for a generic $m\in \Z_{>0}$.
 
It is proved in \cite{s2} that this $(m,1)^{\rm th}$ Fourier coefficient is given by\begin{align}&\int_0^1\int_0^1\int_0^1 E_{\mathcal P_{1,1,1}}\biggl(\left(\begin{smallmatrix} 1  &u_1 &u_3 \\
    &   1 & u_2  \\
    & &  1\end{smallmatrix}\right)g,s\biggr) e^{-2\pi i (mu_1+u_2)}\,du_1\,du_2\,du_3\\\nonumber&=\frac{1}{m\cdot\hskip-8pt\prod\limits_{1\le j<\ell\le 3}\zeta^*(1+s_j-s_\ell)}\left(\underset{c_1 c_2c_3=m}{\sum_{c_1,c_2,c_3\in \Z_{>0}}}c_1^{s_1}c_2^{s_2}c_3^{s_3}\right) W^{(3)}_{(s_1,s_2,s_3)}\biggl(\left(\begin{smallmatrix} m  & & \\
    &   1 &   \\
    & &  1\end{smallmatrix}\right)g\biggr),\end{align}where $W^{(3)}_{(s_1,s_2,s_3)}$ denotes the unique rapidly decaying $\GL(3,\R)$ Whittaker function with Langlands parameter $(s_1,s_2,s_3)\in\C^3$ (see \cite{s2}, \cite{s5}).
  It is known (\cite[\S 5.9]{s5}) that $W^{(3)}_{(s_1,s_2,s_3)}$ is invariant under any permutation of  $s_1, s_2, s_3$.  It is also immediate that the divisor sum satisfies the same invariances.
 
The uniqueness follows from Proposition \ref{prop:OneBlock} below.\end{proof}
   
   \begin{remark}\rm Note that the product $\prod\limits_{1\le j<\ell\le 3}\zeta^*(1+s_j-s_\ell)$ is not invariant under permutations of $(s_1,s_2,s_3)$.   It is for this reason that we need to multiply the Eisenstein series $E_{\mathcal P_{1,1,1}}(g,s)$ by this product to obtain our functional equations for $E^*_{\mathcal P_{1,1,1}}(g,s)$. \end{remark}
 \end{example}

 \begin{example} {\bf (\boldmath The Eisenstein series $E_{\mathcal P_{1,2},1\otimes\phi}(g,s)$,  $E_{\mathcal P_{2,1},\phi\otimes 1}(g,s)$ for  $\SL(3,\Z)$).} 
 Here we consider the partitions $3=1+2$ and $3=2+1$. In these cases we are twisting the Eisenstein series $E_{\mathcal P_{1,1,1}}(g,s)$ by a Maass form for $\SL(2,\Z)$.   
 Here, the 1 appearing in the notation $1\otimes\phi$ denotes the constant function $1$ on the upper left $1\times 1$ block of our $3\times 3$ matrix, and $\phi$ denotes a Maass form for $\SL(2,\Z)$, which is a function on the lower right $2\times 2$ block.  Similarly,  the notation $\phi\otimes1$ references a Maass form $\phi$ as a function on  the upper left  $2\times 2$ block of our $3\times 3$ matrix, and  the constant $1$ as a function on the lower right $1\times 1$ block.
 
 We first consider the partition $3=1+2$, represented by $\mathcal P_{1,2}$.  The power function in this case takes the following form: let $g=dxyk$, with $d$ a central element of $\GL(3,\R)$ and $k\in {\rm O}(3,\R)$. Then for $s=(s_1,s_2)\in\C^2$ with $s_1+2s_2=0$, we have
  $$\big|g\big|^s_{\mathcal P_{1,2}}:=\left|  \left(\begin{matrix} y_1y_2  & 0&0 \\
  0  &   y_1 &0   \\
 0   & 0&  1\end{matrix}\right)\right|^s_{\mathcal P_{1,2}}=(y_1 y_2)^{s_1}\left| \det \left(\begin{matrix} y_1 & 0 \\
 0   &   1    \end{matrix}\right)\right|^{s_2}=y_1^{s_1+s_2} y_2^{s_1}. $$Suppose $\phi$ is a Maass form for $\SL(2,\Z)$.  Associated to $\phi$ is a Langlands parameter $\al=(\al_1,\al_2)\in\C^2$, where $\al_1+\al_2=0$ and $\tfrac{1}{4}-\al_1^2=\tfrac{1}{4}-\al_2^2$ is the Laplace eigenvalue of $\phi$.  
    
By the Iwasawa decomposition, every $g\in\GL(3,\R)$ can be written in the form $g=\left(\begin{smallmatrix}\mathfrak m_1 (g)&* \\
    &   \mathfrak m_2(g)   \end{smallmatrix}\right)k$ for some $k\in {\rm O}(3,\R)$, where  $\mathfrak m_1(g)\in \GL(1,\R)\cong \R^\times$  and $\mathfrak m_2(g)\in \GL(2,\R)$.  Then we define the Eisenstein series\begin{equation}\label{EisDef12}E_{\mathcal P_{1,2},1\otimes \phi} (g,s)=\sum_{\gamma\in \left(\begin{smallmatrix} *  & *&* \\
    &   * &*   \\
    &* &  *\end{smallmatrix}\right)\big\backslash \SL(3,\Z)}\phi\big(\mathfrak m_2(\gamma g)\big)\,\big|\mathfrak m_1(\gamma g)\big|^{s_1+1}\, \big|\det \mathfrak m_2(\gamma g)\big|^{s_2-1/2}.\end{equation}
    
    Next we consider  the partition $3=2+1$ represented by $\mathcal P_{2,1}$.  The power function in this case takes the following form: as before, let $g=dxyk$, with $d$ a central element of $\GL(3,\R)$ and $k\in {\rm O}(3,\R)$.  Then for $s=(s_1,s_2)\in\C^2$ with $2s_1+s_2=0$, we have
  $$\big|g\big|^s_{\mathcal P_{2,1}}:=\left|  \left(\begin{matrix} y_1y_2  &0 &0 \\
 0   &   y_1 &  0 \\
   0 &0&  1\end{matrix}\right)\right|^s_{\mathcal P_{2,1}}= \left| \det \left(\begin{matrix} y_1 y_2 & 0 \\
  0  &   y_1    \end{matrix}\right)\right|^{s_1} =y_1^{2s_1 } y_2^{s_1}. $$Suppose $\phi$ is a Maass form for $\SL(2,\Z)$ with Langlands parameter $\al=(\al_1,\al_2)\in\C^2$, where $\al_1+\al_2=0$. By the Iwasawa decomposition, every $g\in\GL(3,\R)$ can be written in the form $g=\left(\begin{smallmatrix}\mathfrak m_1 (g)&* \\
    &   \mathfrak m_2(g)   \end{smallmatrix}\right)k$ for some $k\in {\rm O}(3,\R)$, where  $\mathfrak m_1(g)\in \GL(2,\R)$  and $\mathfrak m_2(g)\in \GL(1,\R)\cong \R^\times$.  Then we define the Eisenstein series\begin{equation}\label{P21Eisdef}E_{\mathcal P_{2,1},\phi\otimes 1} (g,s):=\sum_{\gamma\in \left(\begin{smallmatrix} *  & *&* \\
  *  &   * &*   \\
    & &  *\end{smallmatrix}\right)\big\backslash \SL(3,\Z)}\phi\big(\mathfrak m_1(\gamma g)\big)\,\big|\mathfrak m_1(\gamma g)\big|^{s_1+1/2}\, \big|\det \mathfrak m_2(\gamma g)\big|^{s_2-1}.\end{equation}

 Recall that the L-function $L(s,\phi)$ and its completion $L^*(s,\phi)$ satisfy (see  \cite[Theorem 3.15.3]{s5}) the functional equation
\begin{align*}L^*(s,\phi)&:=\pi^{-s}\Gamma(\tfrac{s+\alpha_1+\epsilon}{2})\Gamma(\tfrac{s+\alpha_2+\epsilon}{2})L(s,\phi)=(-1)^\epsilon L^*(1-s,\phi),\end{align*} 
where $\epsilon=0$ if $\phi$ is an even Maass form and $\epsilon=1$ if $\phi$ is odd.

    \begin{propositionnt} Let  $E_{\mathcal P_{1,2},1\otimes \phi}$ and $E_{\mathcal P_{2,1},  \phi\otimes1}$ be as in equations  
    \eqref{EisDef12} and \eqref{P21Eisdef}.  Define
   \begin{align*}E^*_{\mathcal P_{1,2},1\otimes\phi}(g,s)&:= L^*(1+s_2-s_1,\phi)E_{\mathcal P_{1,2},1\otimes\phi}(g,s),\qquad \big(s=(s_1,s_2)\in\C^2,\ s_1+2s_2=0\big),\\E^*_{\mathcal P_{2,1},\phi\otimes1}(g,s)&:= L^*(1+s_2-s_1,\phi)E_{\mathcal P_{2,1},\phi\otimes1}(g,s),\qquad \big(s=(s_1,s_2)\in\C^2,\ 2s_1+s_2=0\big).\end{align*}
Then the functional equation takes the form
   $$E^*_{\mathcal P_{1,2},1\otimes\phi}\big(g,(s_1,s_2)\big)=E^*_{\mathcal P_{2,1},\phi\otimes1}\big(g,(s_2,s_1)\big)$$for all $s_1,s_2\in\C$.   Moreover, this functional equation is unique in that, if $\mu$ is any real linear transformation of $s$ such that  $$E^*_{\mathcal P_{1,2},1\otimes\phi}\big(g,(s_1,s_2)\big)=E^*_{\mathcal P_{2,1},\phi\otimes1}\big(g,\mu(s)\big),$$then $\mu(s)$ is a permutation of $s_1,s_2$.

      \begin{remark}\rm  We note that in  $E^*_{\mathcal P_{1,2},1\otimes\phi}\big(g,(s_1,s_2)\big)$ it is a requirement that $s_1+2s_2=0$.  Similarly, the definition of $E^*_{\mathcal P_{2,1},\phi\otimes1}\big(g,(s_2,s_1)\big)$ requires that $2s_2+s_1=0$, which is, of course, the same condition.  \end{remark}
   \end{propositionnt}

 \begin{proof} The meromorphic continuation of the Eisenstein series and the functional equation 
   $$E^*_{\mathcal P_{1,2},1\otimes\phi}\big(g,(s_1,s_2)\big)=E^*_{\mathcal P_{2,1},\phi\otimes1}\big(g,(s_2,s_1)\big)$$were proved by Langlands \cite{s9}.  We will determine this functional equation heuristically by considering just the $(m,1)^{\rm th}$ Fourier coefficients of these Eisenstein series.

  Recall  the definition of the adjoint L-function:  $L(s, \Ad{\phi}) := L(s, \phi\times\overline{\phi})/\zeta(s)$ where $L(s, \phi\times\overline{\phi})$ is the Rankin-Selberg convolution L-function as in \S12.1 of 
\cite{s5}.  
The completed adjoint L-function at $s=1$ is  given by
$$L^*(1,\Ad\phi):=\Gamma(\tfrac{1}{2}+\alpha_1)\Gamma(\tfrac{1}{2}+\alpha_2)L(1,\Ad \phi).$$

It is proved in \cite{s3} that, if $s_1+2s_2=0$, then the $(m,1)^{\rm th}$  Fourier coefficient of $E_{\mathcal P_{1,2},1\otimes \phi}$ is given by\begin{align*} &\int_0^1\int_0^1\int_0^1 E_{\mathcal P_{1,2},1\otimes \phi}\biggl(\left(\begin{smallmatrix} 1  &u_1 &u_3 \\
    &   1 & u_2  \\
    & &  1\end{smallmatrix}\right)g,(s_1,s_2)\biggr) e^{-2\pi i (mu_1+u_2)}\,du_1\,du_2\,du_3\\\nonumber&= \frac{1}{m L^*(1,\Ad \phi)^{1/2} L^*(1+s_2-s_1,\phi)}\left (\underset{c_1 c_2 =m}{\sum_{c_1,c_2 \in \Z_{>0}}}\lambda_\phi(c_1)c_1^{s_1}c_2^{s_2} \right) W^{(3)}_{(s_1,s_2+\alpha_1,s_2+\alpha_2)}\biggl(\left(\begin{smallmatrix} m  & & \\
    &   1 &   \\
    & &  1\end{smallmatrix}\right)g\biggr).\end{align*}We therefore have
    \begin{align}\label{12Fourier}&\int_0^1\int_0^1\int_0^1 E^*_{\mathcal P_{1,2},1\otimes \phi}\biggl(\left(\begin{smallmatrix} 1  &u_1 &u_3 \\
    &   1 & u_2  \\
    & &  1\end{smallmatrix}\right)g,(s_1,s_2)\biggr) e^{-2\pi i (mu_1+u_2)}\,du_1\,du_2\,du_3\\\nonumber&\hskip40pt= \frac{1}{m L^*(1,\Ad \phi)^{1/2}  }\left (\underset{c_1 c_2 =m}{\sum_{c_1,c_2 \in \Z_{>0}}}\lambda_\phi(c_1)c_1^{s_1}c_2^{s_2} \right) W^{(3)}_{(s_1,s_2+\alpha_1,s_2+\alpha_2)}\biggl(\left(\begin{smallmatrix} m  & & \\
    &   1 &   \\
    & &  1\end{smallmatrix}\right)g\biggr).\end{align}
    Similarly, it is proved in \cite{s3} that,  if $2s_1+s_2=0$, then the $(m,1)^{\rm th}$  Fourier coefficient of $E^*_{\mathcal P_{2,1}, \phi\otimes1}$ is given by\begin{align}\label{21Fourier}&\int_0^1\int_0^1\int_0^1 E^*_{\mathcal P_{2,1},\phi\otimes 1}\biggl(\left(\begin{smallmatrix} 1  &u_1 &u_3 \\
    &   1 & u_2  \\
    & &  1\end{smallmatrix}\right)g,(s_1,s_2)\biggr) e^{-2\pi i (mu_1+u_2)}\,du_1\,du_2\,du_3\\\nonumber&\hskip40pt= \frac{1}{m L^*(1,\Ad \phi)^{1/2} }\left (\underset{c_1 c_2 =m}{\sum_{c_1,c_2 \in \Z_{>0}}}\lambda_\phi(c_2)c_1^{s_1}c_2^{s_2} \right) W^{(3)}_{(s_1 +\alpha_1,s_1+\alpha_2,s_2)}\biggl(\left(\begin{smallmatrix} m  & & \\
    &   1 &   \\
    & &  1\end{smallmatrix}\right)g\biggr).\end{align} 
  Note that, if we interchange $s_1$ and $s_2$ in the divisor sum in \eqref{12Fourier}, we get the divisor sum appearing in \eqref{21Fourier}.  Also, this interchange sends $W^{(3)}_{(s_1,s_2+\alpha_1,s_2+\alpha_2)}$ to $W^{(3)}_{(s_2,s_1+\alpha_1,s_1+\alpha_2)}$, which equals $W^{(3)}_{(s_1 +\alpha_1,s_1+\alpha_2,s_2)}$,  since $W^{(3)}_{(a,b,c)}$ is invariant under any permutation of $(a,b,c)$.  Finally, this switch  transforms the condition $s_1+2s_2=0$ to the condition $2s_1+s_2=0$.      
  
  The uniqueness now follows from Proposition \ref{prop:IfMuIsLinear} below.
 \end{proof} 

     \end{example}

 \begin{example}{\bf (\boldmath The Eisenstein series $E_{\mathcal P_{2,2},\phi_1\otimes\phi_2}(g,s)$  for  $\SL(4,\Z)$).} 
 Here we consider the partition $4=2+2$. In this case our construction involves a twist by two Maass forms for $\SL(2,\Z)$: a Maass form $\phi_1$  with Langlands parameter $(\al_{1,1},\al_{1,2})\in \C^2$ with $\al_{1,1}+\al_{1,2}=0$, and a Maass form $\phi_2$  with Langlands parameter $(\al_{2,1},\al_{2,2})\in \C^2$ with $\al_{2,1}+\al_{2,2}=0$.   Here, $\tfrac{1}{4}-\al_{j,1}^2 $ is the Laplace eigenvalue of $\phi_j$, for $j=1,2$.
 
  The power function in this case takes the following form: let $g=dxyk$, with $d$ a central element of $\GL(4,\R)$ and $k\in {\rm O}(4,\R)$. Then for $s=(s_1,s_2)\in\C^2$ with $2s_1+2s_2=0$, we have
\begin{align*}\big|g\big|^s_{\mathcal P_{2,2}}&:=\left|  \left(\begin{matrix} y_1y_2y_3  & 0&0&0 \\0& y_1y_2  & 0&0 \\
 0& 0  &   y_1 &0   \\
0& 0   & 0&  1\end{matrix}\right)\right|^s_{\mathcal P_{2,2}}  =\left| \det \left(\begin{matrix} y_1y_2 y_3 & 0 \\
0&y_1 y_2        \end{matrix}\right)\right|^{s_1}\cdot \left| \det \left(\begin{matrix} y_1 & 0 \\
 0   &   1    \end{matrix}\right)\right|^{s_2}\\&\\&=y_1^{2s_1+s_2} y_2^{2s_1}y_3^{s_1}. \end{align*} 
    
By the Iwasawa decomposition, every $g\in\GL(4,\R)$ can be written in the form $g=\left(\begin{smallmatrix}\mathfrak m_1 (g)&* \\
    &   \mathfrak m_2(g)   \end{smallmatrix}\right)k$ for some $k\in {\rm O}(4,\R)$, where  $\mathfrak m_1(g), \mathfrak m_2(g)\in \GL(2,\R)$.  Then we define the Eisenstein series\begin{equation}\label{EisDef22}E_{\mathcal P_{2,2},\phi_1\otimes \phi_2} (g,s)\hskip8pt=\hskip-18pt\sum_{\gamma\in \left(\begin{smallmatrix} 
    *  &*  & *&* \\
    *  &*  &*  & *\\
    &    &*  &*  \\*  &
    &* &  *\end{smallmatrix}\right)\Big\backslash \SL(4,\Z)} \hskip-20pt \phi_1 \big(\mathfrak m_1(\gamma g)\big)\,  \phi_2 \big(\mathfrak m_2(\gamma g)\big) \cdot \big  |   \gamma g\big|^{s+(1,-1)}_{\mathcal P_{2,2}}.\end{equation}
    
 Recall the Rankin-Selberg L-function $L^*(s,\phi_1\times \phi_2)$ as defined in Chapter 12 of \cite{s5}.    The completed L-function for this is given by
\begin{align*}L^*(s,\phi_1\times \phi_2)&:=\pi^{-2s }\left(\,\prod_{j,k=1}^2\Gamma\Big(\frac{s+\alpha_{1,j}+\al_{2,k}}{2}\Big)\right)\cdot L(s,\phi_1\times\phi_2).\end{align*} 

    \begin{propositionnt} Let  $E_{\mathcal P_{2,2},\phi_1\otimes \phi_2}$ be as in equation
    \eqref{EisDef22}.  Define
   \begin{align*}E^*_{\mathcal P_{2,2},\phi_1\otimes\phi_2}(g,s)&:= L^*(1+s_2-s_1,\phi_1\times \phi_2)E_{\mathcal P_{2,2},\phi_1\otimes\phi_2}(g,s) ,\end{align*}
where  $s=(s_1,s_2)\in\C^2$ satisfies $2s_1+2s_2=0$. Then the functional equation takes the form
   $$E^*_{\mathcal P_{2,2},\phi_1\otimes\phi_2}\big(g,(s_1,s_2)\big)=E^*_{\mathcal P_{2,2},\phi_2\otimes\phi_1}\big(g,(s_2,s_1)\big)$$
for all $s_1,s_2\in\C$.  Moreover, this functional equation is unique in that, if if $\sigma\in S_2$ and $\mu$ is any real linear transformation of $s$ such that  $$E^*_{\mathcal P_{2,2},\phi_1\otimes\phi_2}\big(g,(s_1,s_2)\big)=E^*_{\mathcal P_{2,2},\phi_{\sigma(1)}\otimes\phi_{\sigma(2)}}\big(g,\mu(s)\big),$$then $\mu=\sigma$.
\end{propositionnt}
   
\begin{remark}\rm  In  $E^*_{\mathcal P_{2,2},\phi_1\otimes\phi_2}\big(g,(s_1,s_2)\big)$, as required we have $2s_1+2s_2=0$.  Similarly, in  $E^*_{\mathcal P_{2,2},\phi_2\otimes\phi_1}\big(g,(s_2,s_1)\big)$, as required we have $2s_2+2s_1=0$, which is, of course, the same condition.  \end{remark}

 \begin{proof} The meromorphic continuation of the Eisenstein series and the functional equation 
 $$E^*_{\mathcal P_{2,2},\phi_1\otimes\phi_2}\big(g,(s_1,s_2)\big)=E^*_{\mathcal P_{2,2},\phi_2\otimes\phi_1}\big(g,(s_2,s_1)\big)$$were proved by Langlands \cite{s9}.  We will determine this  functional equation  heuristically by considering just the $(m,1,1)^{\rm th}$ Fourier coefficient of this Eisenstein series.
   
It is proved in \cite{s3} that, if $2s_1+2s_2=0$, then the $(m,1,1)^{\rm th}$  Fourier coefficient of $E^*_{\mathcal P_{2,2},\phi_1\otimes \phi_2}$ is given by\begin{align}\label{22Fourier-a}&\int_0^1\cdots \int_0^1 E^*_{\mathcal P_{2,2},\phi_1\otimes \phi_2}\biggl(\left(\begin{smallmatrix} 1  &u_1 &u_4& u_6 \\ &1  &u_2 &u_5 \\
   & &   1 & u_3&  \\
%:
    & & & 1\end{smallmatrix}\right)g,(s_1,s_2)\biggr) e^{-2\pi i (mu_1+u_2+u_3)}\,\prod_{i=1}^6 du_i\\\nonumber&\hskip15pt= \frac{\displaystyle\left (\underset{c_1 c_2 =m}{\sum_{c_1,c_2 \in \Z_{>0}}}\lambda_{\phi_1}(c_1)\lambda_{\phi_2}(c_2)c_1^{s_1}c_2^{s_2} \right) W^{(4)}_{(s_1+\alpha_{1,1},s_1+\alpha_{1,2},s_2+\alpha_{2,1},s_2+\alpha_{2,2})}\biggl(\biggl(\begin{smallmatrix} m  & & &\\ &1  & & \\
   & &   1 &   \\
 &   & &  1\end{smallmatrix}\biggr)g\biggr)}{m^{3/2} L^*(1,\Ad \phi_1)^{1/2}  L^*(1,\Ad \phi_2)^{1/2}  }.\end{align}
Interchanging $\phi_1$ and $\phi_2$ we find that, if $2s_1+2s_2=0$, then the $(m,1,1)^{\rm th}$  Fourier coefficient of $E_{\mathcal P_{2,2}, \phi_2\otimes\phi_1}$ is given by\begin{align}\label{22Fourier-b}&\int_0^1\cdots \int_0^1 E^*_{\mathcal P_{2,2},\phi_2\otimes \phi_1}\biggl(\left(\begin{smallmatrix} 1  &u_1 &u_4& u_6 \\ &1  &u_2 &u_5 \\
   & &   1 & u_3&  \\
%:
    & & & 1\end{smallmatrix}\right)g,(s_1,s_2)\biggr) e^{-2\pi i (mu_1+u_2+u_3)}\,\prod_{i=1}^6 du_i\\\nonumber&\hskip15pt= \frac{\displaystyle\left (\underset{c_1 c_2 =m}{\sum_{c_1,c_2 \in \Z_{>0}}}\lambda_{\phi_2}(c_1)\lambda_{\phi_1}(c_2)c_1^{s_1}c_2^{s_2} \right) W^{(4)}_{(s_1+\alpha_{2,1},s_1+\alpha_{2,2},s_2+\alpha_{1,1},s_2+\alpha_{1,2})}\biggl(\biggl(\begin{smallmatrix} m  & & &\\ &1  & & \\
   & &   1 &   \\
 &   & &  1\end{smallmatrix}\biggr)g\biggr)}{m^{3/2} L^*(1,\Ad \phi_1)^{1/2}  L^*(1,\Ad \phi_2)^{1/2}  }.\end{align}

Note that, if we interchange $s_1$ and $s_2$ in the divisor sum in \eqref{22Fourier-a}, we get the divisor sum appearing in \eqref{22Fourier-b}.  Also, this interchange sends the Whittaker function with Langlands parameter\begin{equation*} {(s_1+\alpha_{1,1},s_1+\alpha_{1,2},s_2+\alpha_{2,1},s_2+\alpha_{2,2})}\end{equation*}    to the Whittaker function with  Langlands parameter
\begin{equation}\label{tuple2} {(s_2+\alpha_{1,1},s_2+\alpha_{1,2},s_1+\alpha_{2,1},s_1+\alpha_{2,2})}. \end{equation} The Whittaker function with Langlands parameter given by \eqref{tuple2}  equals the Whittaker function in equation \eqref{22Fourier-b} because the Whittaker function is invariant under permutations of the Langlands parameters.

  The uniqueness now follows from Proposition \ref{prop:IfMuIsLinear} below.
\end{proof}

     \end{example}
 
\section{\large\boldmath  \bf Eisenstein series for a parabolic subgroup of $\GL(n,\R)$}
\label{EisSeriesDef}

\begin{definition}{\bf (Parabolic Subgroup).}\label{GLnParabolic} For $n\ge 2$ and $1\le r\le n,$ consider a partition of $n$ given by
$n = n_1+\cdots +n_r$ with positive integers $n_1,\cdots,n_r.$  We define the {\it standard parabolic subgroup} $$\mathcal P := \mathcal P_{n_1,n_2, \ldots,n_r} := \left\{\left(\begin{matrix} \GL(n_1) & * & \cdots &*\\
0 & \GL(n_2) & \cdots & *\\
\vdots & \vdots & \ddots & \vdots \\
0 & 0 &\cdots & \GL(n_r)\end{matrix}\right)\right\}.$$

Letting $I_r$ denote the $r\times r$ identity matrix, the subgroup
$$N^{\mathcal P} := \left\{\left(\begin{matrix} I_{n_1} & * & \cdots &*\\
0 & I_{n_2} & \cdots & *\\
\vdots & \vdots & \ddots & \vdots \\
0 & 0 &\cdots & I_{n_r}\end{matrix}\right)\right\}$$
is the unipotent radical of $\mathcal P$.  The subgroup
$$M^{\mathcal P} := \left\{\left(\begin{matrix} \GL(n_1) & 0 & \cdots &0\\
0 & \GL(n_2) & \cdots & 0\\
\vdots & \vdots & \ddots & \vdots \\
0 & 0 &\cdots & \GL(n_r)\end{matrix}\right)\right\}$$
 is the standard choice of Levi subgroup of $\mathcal P$.
\end{definition}

 \begin{definition}{\bf (Automorphic form $\Phi$ associated to a parabolic $\mathcal P$).} \label{def:maassparabolic} Let $n\ge 2$. Consider a partition $n = n_1+\cdots +n_r$ with $1 < r < n$. Let  $\mathcal P := \mathcal P_{n_1,n_2, \ldots,n_r} \subseteq \GL(n,\mathbb R).$
  For $i = 1,2,\ldots, r$, let
$\phi_i:\GL(n_i,\mathbb R)\to\mathbb C$ be either the constant function 1 (if $n_i=1$) or a Maass cusp form for $\SL(n_i,\mathbb Z)$ (if $n_i>1$).  {\it The form} $\Phi := \phi_1\otimes \cdots\otimes \phi_r$ is defined on $\GL(n,\mathbb R)=\mathcal P(\mathbb R)K$ (where $K= \text{\rm O}(n,\mathbb R))$ by the formula
$$\Phi(u \mathfrak m k ) := \prod_{i=1}^r \phi_i(\mathfrak  m_i), \qquad (u\in  N^{\mathcal P}, \mathfrak m\in  M^{\mathcal P},k\in K)$$
where
  $\mathfrak m \in M^{\mathcal P}$ has the form
$\mathfrak m = \left(\begin{smallmatrix} \mathfrak m_1 & 0 & \cdots &0\\
0 & \mathfrak m_2 & \cdots & 0\\
\vdots & \vdots & \ddots & \vdots \\
0 & 0 &\cdots & \mathfrak m_r\end{smallmatrix}\right)$, with    $\mathfrak m_i\in \GL({n_i},\mathbb R).$  In fact, this construction works equally well if some or all of the $\phi_i$ are Eisenstein series.

 \end{definition}

\begin{definition}{\bf (Power function for a parabolic subgroup).}  \rm\label{ParabolicNorm}  Let $n\ge 2$ and $2\le r\le n$.  Fix a partition
$n = n_1+n_2 + \cdots +n_r$ with associated parabolic subgroup $\mathcal P := \mathcal P_{n_1,n_2, \ldots,n_r}.$   
Let $s = (s_1,s_2, \ldots, s_r) \in\mathbb C^r$ satisfy
$\sum\limits_{i=1}^r n_i s_i = 0.$

   For $g\in \mathcal P$, with diagonal block entries $\mathfrak m_i\in \GL(n_i,\mathbb R)$, we define the {\it power function}
 \vskip-10pt
 $$| g  |_{_{\mathcal P}}^s:=\prod_{i=1}^r\left| \text{\rm det}(\mathfrak m_i)\right|^{s_i}.$$Since $\det (\mathfrak m_i k_i)=\pm \det ( \mathfrak m_i )$ for $k_i\in{\rm O}(n_i,\R)$, we see that $| g  |_{_{\mathcal P}}^s$ is invariant under right multiplication by ${\rm O}(n,\R)$.  Because of this, and the fact that $\sum\limits_{i=1}^r n_i s_i = 0$, it follows that $| g  |_{_{\mathcal P}}^s$ extends to a well-defined function on $\GL(n,\R)$, invariant by the center and right multiplication by ${\rm O}(n,\R)$.
 \end{definition}

\begin{definition}{\bf (\boldmath $\rho$-function for a parabolic).} \label{rhofunction} Let $n\ge 2$ and $2\le r\le n$.  Let $\mathcal P$ be the parabolic subgroup $\mathcal P := \mathcal P_{n_1,n_2, \ldots,n_r}.$ Then we define\begin{equation*}
\rho_{_{\mathcal P}}(j):=\left\{
                            \begin{array}{ll}
                              \frac{n-n_1}{2}, & j=1, \\
                              \frac{n-n_j}{2}-n_1-\cdots-n_{j-1}, & j\ge 2,
                            \end{array}
                          \right.
\end{equation*}and  $$\rho_{_{\mathcal P}}=(\rho_{_{\mathcal P}}(1),\rho_{_{\mathcal P}}(2),\ldots,\rho_{_{\mathcal P}}(r)).$$ \end{definition}

\begin{remark} \rm The $\rho$-function is introduced as a normalizing factor (a shift in the $s$ variable) in the definition of Eisenstein series below, to make later formulae as simple as possible. Note that, in the definition of Eisenstein series for  reductive groups,  $\rho_{_{\mathcal P}}$ equals half the sum of the roots of the parabolic subgroup. \end{remark}

 \begin{definition} {\bf  (\boldmath Langlands Eisenstein series).} \label{LangEisDef}
 Let $n\ge 2$ and $2\le r\le n$.  Let $\mathcal P$ be the parabolic subgroup $\mathcal P := \mathcal P_{n_1,n_2, \ldots,n_r}.$ Suppose $\Phi$ is a Maass form associated to $\mathcal P$,  as in Definition \ref{def:maassparabolic}.  Let $s = (s_1,s_2, \ldots, s_r) \in\mathbb C^r$ with 
$\sum\limits_{i=1}^r n_i s_i = 0$ and let $| g|^{s}_{\mathcal P}$ be the power function as in Definition \ref{ParabolicNorm}.  

For  $\Gamma_n:=\SL(n,\Z)$,   we define the {\it Eisenstein series} 
\begin{equation}E_{\mathcal P,\Phi}(g,s):= \sum_{\gamma\in (\Gamma_n\cap \mathcal P) \backslash \Gamma_n}\Phi(
\gamma g)\cdot |\gamma g|^{s+\rho_{_{\mathcal P}}}_{\mathcal P},\label{EisDef}\end{equation}
with $\rho_{_{\mathcal P}}$ as in Definition \ref{rhofunction}.\end{definition}

\begin{example} {\bf (Borel Eisenstein series).} \rm Consider the partition  $$n=\underbrace{1+1+\cdots+1}_{n\  {\rm times}},$$  associated to the Borel parabolic subgroup $$\mathcal B:= \mathcal P_{ 1,1,\ldots,1} := \left\{\left(\begin{matrix} * & * & \cdots &*\\
0 & * & \cdots & *\\
\vdots & \vdots & \ddots & \vdots \\
0 & 0 &\cdots & *\end{matrix}\right)\right\}.$$Then the Borel Eisenstein series is constructed as follows.  Let $\Phi=1$ be the trivial function, and choose $s=(s_1,s_2,\ldots, s_{n})$ where $\sum\limits_{i=1}^n s_i=0$. Then 
$\rho_{_{\mathcal B}}=\bigl(\frac{n-1}{2},\frac{n-3}{2},\ldots,\frac{1-n}{2}\bigr),$ and \eqref{EisDef} becomes
\begin{equation}E_{\mathcal B}(g,s)=\sum_{\gamma\in (\Gamma_n\cap \mathcal B) \backslash \Gamma_n}|\gamma g|^{s+\rho_{_{\mathcal B}}}_{\mathcal B}.\label{BorelDef}\end{equation}The Borel Eisenstein series for $\SL(2,\Z)$ and $\SL(3,\Z)$ are given in Examples \ref{sl2Borel} and \ref{sl3Borel}  respectively. \end{example}

\section{\large  \bf Whittaker function for Langlands Eisenstein series}

\begin{definition} {\bf (Jacquet's Whittaker function).}\rm\label{def:JacWhittFunction} For $n\geq 2$,  let $\alpha = (\alpha_1,\alpha_2, \;\ldots, \;\alpha_n)\in \C^n$ with $\sum\limits_{i=1}^n \al_i=0$; let $g\in \GL(n,\R)$.  We define the {\it completed    Whittaker function} $W_{\alpha}: \GL(n,\mathbb R)\big/\left(
  \text{O}(n,\mathbb R)\cdot \mathbb R^\times\right) \to \mathbb C$, {\it with Langlands parameter} $\alpha$,   by the integral\begin{equation}\label{Whitdef}W^{(n)}_\alpha(g)= \prod_{1\leq j< k \leq n} \frac{\Gamma\big(\frac{1 + \alpha_j - \alpha_k}{2}\big)}{\pi^{\frac{1+\alpha_j - \alpha_k}{2}} }\cdot \int\limits_{U_n(\mathbb R)}\big|w_{n} \cdot ug\big|^{s+\rho_{_{\mathcal B}}}_{\mathcal B} \,\overline{\psi_{1, 1,\ldots,   1}(u)} \, du,   \end{equation}where $w_n$ is the long element of the Weyl group for $\GL(n,\R)$, and $\lvert * \rvert_{\mathcal B}^s$ is the power function for the Borel $\mathcal{B}$ given in Definition~\ref{ParabolicNorm}.
This integral converges absolutely if $\re(\al_i-\al_{i+1})>0$ for $1\le i\le n-1$, has meromorphic continuation to all $\alpha\in \C^n$ satisfying $\sum\limits_{i=1}^n \al_i=0$, and is invariant under permutations of $\al_1,\al_2,\ldots,\al_n$  (cf.  \cite[\S3]{s3}).
 \end{definition}

\begin{propositionnt} \label{langlandsparamsforPhi} Let ${\mathcal P}:=\mathcal P_{n_1,n_2,\ldots n_r}$ be a parabolic subgroup of $\GL(n,\R)$, and let $\Phi=\phi_1\otimes \phi_2\otimes \cdots\otimes \phi_r$ be an automorphic form associated to $\mathcal P$. For $j=1,2,\ldots,r,$ let $ (\alpha_{j,1},\ldots,\alpha_{j,n_j})$ denote the Langlands parameter of $\phi_j.$ We adopt the convention that if $n_j=1$ then $\alpha_{j,1} = 0.$
Then the Langlands parameter of $E_{\mathcal P, \Phi}(g,s)$ (denoted $\alpha_{_{\mathcal P,\Phi}}(s)$) is
\begin{align*}
\alpha_{_{\mathcal P,\Phi}}(s) &=\bigg (\overbrace{\alpha_{1,1}+s_1, \;\ldots \;,\alpha_{1,n_1}+s_1}^{n_1 \;\,\text{\rm terms}}, \;\overbrace{\alpha_{2,1}+s_2, \;\ldots \;,\alpha_{2,n_2}+s_2}^{n_2 \;\,\text{\rm terms}},\\\nn&\hskip200pt
\ldots \quad
,\;\overbrace{\alpha_{r,1}+s_r, \;\ldots \;,\alpha_{r,n_r}+s_r}^{n_r \;\,\text{\rm terms}}\bigg). 
\end{align*}\end{propositionnt}

\begin{proof}
Let $\mathcal{B}$ be the Borel parabolic subgroup of $\GL(n)$.  We need to show that $E_{\mathcal{P},\Phi}(g,s)$ has the same eigenvalues as the power function $\big| * \big|^{\alpha_{\mathcal{P},\Phi}(s)+\rho_{_B}}_{\mathcal B}$.

By construction, the Eisenstein series $E_{\mathcal P,1}(g,s)$,  where the ``1'' in the subscript denotes the constant function 1,  has the same eigenvalues as the power function $\big| * \big|_{\mathcal P}^{s+\rho_{_{\mathcal P}}}$. If we define $s^*+\rho_{_{\mathcal P}}^*\in\C^r$ by
 \begin{align*}
    s^*+\rho_{_{\mathcal P}}^*&:=\biggl(\overbrace{s_1+\rho_{_{\mathcal P}}(1),\ldots,s_1+\rho_{_{\mathcal P}}(1)}^{n_1\ {\rm terms}},\;\overbrace{s_2+\rho_{_{\mathcal P}}(2),\ldots,s_2+\rho_{_{\mathcal P}}(2)}^{n_2\ {\rm terms}} ,\\\nn&\hskip200pt \ldots\quad,\;\overbrace{s_r+\rho_{_{\mathcal P}}(r),\ldots,s_r+\rho_{_{\mathcal P}}(r)}^{n_r\ {\rm terms}}\biggr),
    \end{align*}
then $\big| * \big|_{\mathcal P}^{s+\rho_{_{\mathcal P}}}=\big| * \big|_{\mathcal B}^{s^*+\rho_{_{\mathcal P}}^*}$.  To see this, note that it suffices to check this equality on diagonal matrices ${\rm diag}(d_1,d_2,\ldots,d_n)$.  This can be verified by a simple calculation.

Let $\mathcal B_j$ be the Borel parabolic subgroup of $\GL(n_j)$.    Let $(\al_{j,1},\al_{j,2},\ldots,\al_{j,n_j})$ be the Langlands parameter of $\phi_j$ (that is, $\phi_j$ has the same eigenvalues as $\big|*\big|^{\al_j+\rho_{_{\mathcal B_j}}}_{\mathcal B_j}$).  The Langlands parameter of $\Phi=\phi_1\otimes \phi_2\otimes \cdots\otimes \phi_r$ is  $$\al_{\Phi}=\biggl(\overbrace{\al_{1,1},\ldots,\al_{1,n_1}}^{n_1 \ {\rm terms}},\;\overbrace{\al_{2,1},\ldots,\al_{2,n_2}}^{n_2 \ {\rm terms}} ,\quad \ldots\quad,\: \overbrace{\al_{r,1},\ldots,\al_{r,n_r}}^{n_r \ {\rm terms}}\biggr).$$
Hence $\Phi$ has the same eigenvalues as $\big|*\big|^{\al_{\Phi}+\rho_{_\Phi }}_{\mathcal B}$, where 
 $$\rho_{_\Phi}=\biggl(\overbrace{\tfrac{n_1-1}{2},\tfrac{n_1-3}{2},\ldots,\tfrac{1-n_1}{2}}^{\rho_{_{\mathcal B_1}}},\;\overbrace{\tfrac{n_2-1}{2},\tfrac{n_2-3}{2},\ldots,\tfrac{1-n_2}{2}}^{\rho_{_{\mathcal B_2}}} ,\quad\ldots \quad,\; \overbrace{\tfrac{n_r-1}{2},\tfrac{n_r-3}{2},\ldots,\tfrac{1-n_r}{2}}^{\rho_{_{\mathcal B_r}}}\biggr).$$
The eigenvalues of $E_{\mathcal P, \Phi}(*,s)$ match those of $\big|*\big|^{\al_{\Phi}+\rho_{_\Phi}}_{\mathcal B}\cdot \big|*\big|^{s^*+\rho_{_{\mathcal P}}^*}_{\mathcal B}$.  It therefore suffices to show that
\begin{equation}\label{macthexp}\al_{\Phi}+\rho_{_\Phi}+s^*+\rho_{_{\mathcal P}}^*=\al_{\mathcal P,\Phi}(s)+\rho_{_{\mathcal B}}.\end{equation}This is equivalent to the identity $\rho_{_\Phi}+\rho_{_{\mathcal P}}^*=\rho_{_{\mathcal B}}$, which is immediate from the definitions.
\end{proof}

\section{\large \bf Statement and proof of the Main Theorem of this paper}
\label{MainTheoremSection}

\begin{proposition} {\bf  (The \boldmath $M^{\rm{th}}$ Fourier coefficient of   $E_{\mathcal P, \Phi}$).}\label{MthEisCoeff}
 Let $$s = (s_1,s_2, \ldots,s_r)\in\mathbb C^r,$$  where $\sum\limits_{i=1}^r n_is_i =0.$
Consider  $E_{\mathcal P, \Phi}(*,s)$ with associated Langlands parameters $\alpha_{_{\mathcal P,\Phi}}(s)$ as defined in Proposition \ref{langlandsparamsforPhi}.  Let $M = (m_1,m_2,\ldots, m_{n-1}) \in \mathbb Z_{>0}^{n-1}$.  Then the $M^{th}$ term in the Fourier-Whittaker expansion of $E_{\mathcal P, \Phi}$ is
\begin{align*}
\int\limits_{0}^1\cdots\int\limits_{0}^1  E_{\mathcal P, \Phi}(ug, s)\, {\rm exp}\left(-2\pi i \sum\limits_{i=1}^{n-1}  m_{i} u_{i,i+1}\right)\prod_{1\le i<j\le n} du_{i,j}\; = \; \frac{{A_{\mathcal P, \Phi}}(M,s)}
{\prod\limits_{k=1}^{n-1}  m_k^{k(n-k)/2}} \; W_{\alpha_{_{\mathcal P,\Phi}}(s)}\big(M g\big),
\end{align*}
where $A_{{\mathcal P, \Phi}}(M,s) = A_{{\mathcal P, \Phi}}\big((1,\ldots,1),s\big) \cdot \lambda_{{\mathcal P, \Phi}}(M,s),$
and
 \begin{align}\label{HeckeCoeff}
    \lambda_{{\mathcal P,\Phi}}\big((m,1,\ldots,1),s\big) & = \hskip-10pt\underset {c_1c_2\cdots c_r = m } {\sum_{    c_1, c_2, \ldots, c_r \, \in \, \mathbb Z_{>0}}} \hskip-10pt \lambda_{\phi_1}(c_1)  \cdots \lambda_{\phi_r}(c_r)
    \cdot c_1^{s_1}  \cdots c_r^{s_r}
\end{align}
 is the $(m,1,\ldots,1)^{th}$ (or more informally  the $m^{th}$) Hecke eigenvalue of $E_{\mathcal P,\Phi}$.
 
 Moreover, suppose $\phi_j$ has Langlands parameter
  $(\alpha_{j,1},\ldots,\alpha_{j,n_j})$, with the convention that if $n_j=1$ then $\alpha_{j,1}=0.$  We also assume that each $\phi_j$ is normalized to have Petersson norm $\langle \phi_j, \phi_j\rangle = 1.$
 Then the first coefficient of $E_{\mathcal P, \Phi}$ is given by
 \begin{align*} 
 & A_{{\mathcal P, \Phi}}\big((1,\ldots,1),s\big) =  \underset{n_k\ne1}{\prod_{k=1}^r} L^*\big(1, \text{\rm Ad}\; \phi_k\big)^{-\frac12 }\hskip-4pt
  \prod_{1\le j<\ell\le r}   L^*\big(1+s_j-s_{\ell}, \;\phi_j\times\phi_{\ell}\big)^{-1}
 \end{align*} 
 up to a non-zero constant factor with absolute value depending  only on $n$. 
Here
$$L^*(1,\,\text{\rm Ad}\;\phi_k) = L(1,\,\text{\rm Ad}\;\phi_k) \prod_{1\le i \ne j\le n_k}  \Gamma\left(\frac{1+\alpha_{k,i}-\alpha_{k,j}}{2}   \right)$$
and
$$L^*(1+s_j-s_\ell, \;\phi_j\times\phi_\ell) = \begin{cases} L^*(1+s_j-s_\ell,\, \phi_j) & \text{if}\; n_\ell =1 \;\text{and}\;n_j\ne 1,\\
L^*(1+s_j-s_\ell,\, \phi_\ell) & \text{if}\; n_j=1 \;\text{and}\;n_\ell\ne 1,\\
\zeta^*(1+s_j-s_\ell) & \text{if}\; n_j=n_\ell=1. 
\end{cases}
$$
Otherwise, $L^*(1+s_j-s_\ell, \phi_j\times\phi_\ell)$ is the completed Rankin-Selberg L-function.
\end{proposition}

\begin{proof}
The proof of this proposition is given in \cite[Theorem 4.12]{s6}. \end{proof} 

  \begin{theorem} {\bf (Main Theorem).} \label{maintheorem}  \it Let ${\mathcal P}:=\mathcal P_{n_1,n_2,\ldots n_r}$ be a parabolic subgroup of $\GL(n,\R)$, and let $\Phi=\phi_1\otimes \phi_2\otimes \cdots\otimes \phi_r$ be an automorphic form associated to $\mathcal P$. For $j=1,2,\ldots,r,$ let $ (\alpha_{j,1},\ldots,\alpha_{j,n_j})$ denote the Langlands parameter of $\phi_j.$ We adopt the convention that if $n_j=1$ then $\alpha_{j,1} = 0.$

Let   $E_{\mathcal P,\Phi}(g,s)$ be as in Definition \ref{LangEisDef}.  Define
   \begin{align*}E^*_{\mathcal P ,\Phi}(g,s)&:=\left(\,\prod_{1\le j<\ell\le r}   L^*\big(1+s_j-s_{\ell}, \;\phi_j\times\phi_{\ell}\big)\right)E_{\mathcal P ,\Phi}(g,s) ,\end{align*}
where  $s=(s_1,s_2,\ldots s_r)\in\C^r$ satisfies $\sum\limits_{j=1}^r n_j s_j=0$.       Then the $M=(m,1,\ldots,1)^{\rm th}$ Fourier-Whittaker coefficient of $E^*_{\mathcal P ,\Phi}(g,s)$, defined by
$$FW_{{\mathcal P, \Phi}}(g,M,s) :=\int\limits_{0}^1\cdots\int\limits_{0}^1  E^*_{\mathcal P, \Phi}(ug, s)\, {\rm exp}\left(-2\pi i \sum\limits_{i=1}^{n-1}  m_{i} u_{i,i+1}\right)\prod_{1\le i<j\le n} du_{i,j},$$where $m_1=m$ and $m_i=1$ for $2\le i\le n-1$,   satisfies the  functional equations
      $$FW_{{\mathcal P, \Phi}}(g,M,s)=FW_{{\sigma\mathcal P, \sigma\Phi}}(g,M,\sigma s)
      $$for any $\sigma\in S_r$.  Here, the action of $\sigma$ on $\mathcal P$, $\Phi$, and $s$ is   given by    
      \begin{equation} \sigma \mathcal  P:= \mathcal P_{n_{\sigma(1)},n_{\sigma(2)},\ldots,n_{\sigma(r)}},\quad
  \sigma \Phi:= \phi_{\sigma(1)}\otimes\phi_{\sigma(2)}\otimes \cdots \phi_{\sigma(r)},\quad
\sigma s:= \big(s_{\sigma(1)},s_{\sigma(2)},\ldots,s_{\sigma(r)}\big).\label{sigma-action}\end{equation}

   \end{theorem}
   
   \begin{proof}By Proposition \ref{MthEisCoeff}, it suffices to show that each of the following three expressions is invariant under the action of any $\sigma\in S_r$:
   $$  \underset{n_k\ne1}{\prod_{k=1}^r} L^*\big(1, \text{\rm Ad}\; \phi_k\big)^{-\frac12 },\quad \underset {c_1c_2\cdots c_r = m } {\sum_{    c_1, c_2, \ldots, c_r \, \in \, \mathbb Z_{>0}}} \hskip-10pt \lambda_{\phi_1}(c_1)  \cdots \lambda_{\phi_r}(c_r)
    \cdot c_1^{s_1}  \cdots c_r^{s_r},\quad W_{\alpha_{_{\mathcal P,\Phi}}(s)}\big(M g\big).$$   The first two of these expressions clearly satisfy these invariances. 
Moreover, by Proposition \ref{langlandsparamsforPhi}, the above action of $\sigma$ amounts to a certain permutation of the coordinates of the Langlands parameter $\alpha_{_{\mathcal P,\Phi}}(s)$.  It is well-known (\cite[\S5.9]{s5}) that the Whittaker function $W_{\alpha_{_{\mathcal P,\Phi}}(s)}$ is invariant under such permutations.\end{proof}

\begin{corollary}  Suppose $\sigma\in S_r$ satisfies $\sigma \mathcal P=\mathcal P$ and $\sigma \Phi=\Phi$.  Then $$FW_{\mathcal P,\Phi}(g,M,s)=FW_{\mathcal P,\Phi}(g,M,\sigma s).$$\end{corollary}

\begin{proof} This is immediate from our Main Theorem.\end{proof}

\begin{examplent}\rm The Borel Eisenstein series satisfies 
$FW_{\mathcal B}(g,M,s)=FW_{\mathcal B}(g,M,\sigma s)$ for any $\sigma\in S_n$.\end{examplent}

\begin{examplent}\rm If $\Phi=\phi_1\otimes \phi_2\otimes\cdots\otimes \phi_r$ and $\phi_k=\phi_j$ for some $1\le k\ne j\le r$, and $\sigma$ is the transposition that interchanges $k$ and $j$, then 
$FW_{\mathcal P,\Phi}(g,M,s)=FW_{\mathcal P, \Phi}(g,M,\sigma s).$\end{examplent}

\begin{theorem} {\bf (Functional equations of \boldmath{$E^*_{\mathcal P,\Phi}(g,s)$}).}\it \label{EisFunctionalEquations}
 Suppose $\sigma\in S_r$ acts on $\mathcal P$, $\Phi$, and $s$ according to \eqref{sigma-action}.
  Then $$\boxed{E^*_{\mathcal P,\Phi}(g,s)=E^*_{\sigma\mathcal P,\sigma\Phi}(g,\sigma s),}$$
  for all $g\in\GL(n,\R).$\end{theorem}
\begin{proof} These functional equations were proved  by Langlands (see \cite{s9}, \cite{s11}).  It is easy to see that every Fourier-Whittaker coefficient of $E^*_{\mathcal P,\Phi}(g,s)$ has to have the same functional equation. We have checked  that the functional equations found by Langlands match the functional equations of the Fourier-Whittaker coefficients given in Theorem \ref{maintheorem}. 
\end{proof}

\section{ \large  \bf Uniqueness of functional equations for Langlands Eisenstein series} \label{UniquenessSection}

\begin{conjecture} \label{conj:UniqueFE}
Let $n\geq 2$.  Suppose that $\mathcal{P}=\mathcal P_{n_1,n_2,\ldots,n_r}$ with $n=n_1+n_2+\cdots+n_r$ with $r\geq 2$ and $\Phi=\phi_1\otimes \phi_2 \otimes \cdots \otimes \phi_r$ with each $\phi_j$ a Maass form for $\SL(n_j,\Z)$ if $n_j\geq 2$ and $\phi_j$ is the constant function one if $n_j=1$.  Let $\sigma'\in S_r$.  Suppose $E^*_{\mathcal P,\Phi}(g,s)$ satisfies a functional equation of the form
\begin{equation}\label{eq:generalUniqueFE}
 E^*_{\mathcal P,\Phi}(g,s) = E^*_{\sigma'\mathcal P,\sigma'\Phi}(g,\mu(s))
\end{equation}
for some affine transformation 
\begin{equation}\label{eq:UniqueFE}
 \mu\left( \left[ \begin{smallmatrix} s_1 \\ s_2 \\ \vdots \\ s_r \end{smallmatrix}\right] \right) = \left[ \begin{smallmatrix} a_{11} & a_{12} & \cdots & a_{1r} \\ a_{21} & a_{22} & \cdots & a_{2r} \\ \vdots & & & \\ a_{11} & a_{12} & \cdots & a_{r1} \end{smallmatrix}\right] \left[ \begin{smallmatrix} s_1 \\ s_2 \\ \vdots \\ s_r \end{smallmatrix} \right] + \left[\begin{smallmatrix} b_1 \\ b_2 \\ \vdots \\ b_r \end{smallmatrix} \right],
\end{equation}
where $a_{ij},b_i\in \R$ for all $1\leq i,j\leq r$.  Then, in fact, $\mu=\sigma$ for some $\sigma\in S_r$ for which $\sigma \mathcal{P} = \sigma'\mathcal{P}$ and $\sigma \Phi = \sigma'\Phi$.
\end{conjecture}

\begin{remark}\rm
There is no loss in generality in assuming that $\Phi$ is in the form 
\begin{equation}\label{eq:UniqueFEiffDivSum}
 \Phi = \overbrace{\phi_1\otimes \phi_1 \otimes \cdots \otimes \phi_1}^{r_1\ {\rm \tiny times}}
 \ \otimes \ 
 \overbrace{\phi_2\otimes \phi_2 \otimes \cdots \otimes \phi_2}^{r_2\ {\rm \tiny times}}
 \ \otimes
\ \cdots\ \otimes \ 
 \overbrace{\phi_N\otimes \phi_N \otimes \cdots \otimes \phi_N}^{r_N\ {\rm \tiny times}},
\end{equation}
where $\phi_i=\phi_j$ if and only if $i=j$, and that  (\ref{eq:generalUniqueFE}) holds with  $\sigma'$ equaling  the identity element.  Indeed, if $\Phi$ is not of this form, then there exists some element $\sigma_0\in S_r$ such that $\sigma_0 \Phi$ is of the form given on the right hand side of (\ref{eq:UniqueFEiffDivSum}).  Then, since Theorem~\ref{EisFunctionalEquations} implies that
 \[ E^*_{\mathcal P,\Phi}(g,s) = E^*_{\sigma_0\mathcal P,\sigma_0\Phi}(g, \sigma_0  s), \]
we have that 
\begin{equation*}
 E^*_{\mathcal P,\Phi}(g,s) = E^*_{\sigma'\mathcal P,\sigma'\Phi}(g,\mu(s)) 
\end{equation*}
if and only if
\begin{equation}\label{EisEqualities} E^*_{\sigma_0\mathcal P,\sigma_0 \Phi}(g,s) = E^*_{\mathcal P,\Phi}(g,   {\sigma_0^{-1} s}) = E^*_{\sigma'\mathcal P,\sigma'\Phi}(g, {\mu(\sigma_0^{-1} s)}) = E^*_{\sigma_0\mathcal P,\sigma_0\Phi}(g,\mu'(s)), \end{equation}
where $\mu'(s) = \sigma_0\,(\sigma')^{-1}\mu\big(\sigma_0s\big)$.  (The third equality in \eqref{EisEqualities} is by Theorem~\ref{EisFunctionalEquations} with $\sigma_0(\sigma')^{-1}$ in place of $\sigma$ and $\mu(\sigma_0^{-1}s)$ in place of $s$.)

 Since $\mu$ is a permutation sending $\Phi$ to $\sigma'\Phi$ if and only if $\mu'$ is a permutation that acts trivially on $\sigma_0\Phi$, we see that \eqref{eq:generalUniqueFE} holds if and only if
 \[ E^*_{\sigma_0\mathcal P,\sigma_0\Phi}(g,s) = E^*_{\sigma_0\mathcal P,\sigma_0\Phi}(g,\mu'(s)), \]
where $\mu'$ is as in \eqref{eq:UniqueFE}.  Finally, the conclusion of the conjecture, i.e., that $\mu$ is a permutation which sends $\Phi$ to $\sigma' \Phi$, means that $\mu'$ fixes $\sigma_0\Phi$.  Note that this implies that $\mu'$ must be of the form $\sigma_1\times \sigma_2 \times \cdots \times \sigma_N$, where each $\sigma_i \in S_{r_i}$ acts on the tensor   product of $r_i$ copies of $\phi_i$.

In the proofs of the special cases given in Propositions \ref{prop:OneBlock} and \ref{prop:IfMuIsLinear} given below, we will assume that we are in this special case: $\Phi$ as in \eqref{eq:UniqueFEiffDivSum} and $\sigma'$ trivial.
\end{remark}

We give the following  Propositions \ref{prop:OneBlock},
\ref{prop:IfMuIsLinear}, as evidence for Conjecture~\ref{conj:UniqueFE}.

\begin{propositionnt}\label{prop:OneBlock}
Conjecture~\ref{conj:UniqueFE} holds in the special case that $\phi_1=\phi_2=\cdots =\phi_r$ is the same automorphic form $\phi$ with $r\geq 2$.  (Note: In this case, every permutation $\sigma \in S_r$ has the property that $\sigma \mathcal{P} = \mathcal{P}$ and $\sigma \Phi = \Phi$.)
\end{propositionnt}

\begin{conjecture}\label{distinct} Fix an integer $N\ge 2.$  Assume $\ell_1,\ell_2,\ldots,\ell_N$ are arbitrary integers greater than one. Let $\eta_1,\eta_2,\ldots, \eta_N$ be distinct Maass forms where each $\eta_j$ is a Maass form for $\SL(\ell_j,\mathbb Z)$ with associated $n^{th}$  Hecke eigenvalue  $\lambda_j(n)$. Then there exists a prime $p$ such that $\lambda_1(p),\lambda_2(p),\ldots, \lambda_N(p)$ are all distinct and non zero.
\end{conjecture}

\begin{remark}\rm Conjecture \ref{distinct} can be proved in the case that $N = 2$ by the method introduced in \cite{s7}.
\end{remark}

\begin{propositionnt}\label{prop:IfMuIsLinear}
Assume Conjecture \ref{distinct}  and assume that $\mu$  given in \eqref{eq:UniqueFE} is linear, i.e., $b_i=0$ for each $i=1,2,\ldots,r.$ 
 Then Conjecture~\ref{conj:UniqueFE} holds.  
\end{propositionnt}

\begin{remark}\rm
The proof uses the fact that if $\phi_j\neq \phi_k$, then we can find $p$ such that $\lambda_j(p)$ and $\lambda_k(p)$ are nonzero and distinct.  If $\mu$ is not necessarily linear and $\Phi$ consists of at least two distinct automorphic forms, our proof breaks down if $\lambda_j(p)=p^b\lambda_k(p)$ for some value of $b$ independent of $p$.
\end{remark}

Before giving the proofs of these propositions, we give a Lemma which reduces the proof of Conjecture~\ref{conj:UniqueFE} to showing that a functional equation for the divisor sum puts strong restrictions on the affine transformation $\mu$.

\begin{lemma}\label{lem:UniqueFEiffDivSum}
Assume that there exist integers $r_1,r_2,\ldots,r_N$ for which $r=r_1+r_2+\cdots +r_N$ and $\Phi$ is as in \eqref{eq:UniqueFEiffDivSum}.  Set $\lambda_j = \lambda_{\phi_j}$, and write \[ \mu_j(s) = a_{j1}s_1+a_{j2}s_2+\cdots + a_{jr}s_r + b_j. \]
Conjecture~\ref{conj:UniqueFE} holds if the following is true: If $\mu$ is an affine transformation for which, setting $\hat{r}_i = \sum\limits_{j=1}^i r_j$,
\begin{equation}\label{eq:DivisorSum}
 \sum_{j=1}^N \lambda_j(p)\left( \sum_{i=\hat{r}_{j-1}+1}^{\hat{r}_j} p^{s_i}\right) =  \sum_{j=1}^N \lambda_j(p)\left( \sum_{i=\hat{r}_{j-1}+1}^{\hat{r}_j} p^{\mu_i(s)}\right)
\end{equation}
is true for all $s=(s_1,s_2,\ldots,s_r) \in \C^r$ satisfying $\sum\limits_{j=1}^r n_js_j=0$ and all primes $p$, then it must be the case that $\mu=\sigma\in S_r$ is of the form $\sigma=\sigma_1\times \sigma_2\times \cdots \times \sigma_N$ for which $\sigma_j \in S_{r_j}$ permutes the $j$-th block of $r_j$ forms $\phi_j \otimes \phi_j \otimes \cdots \otimes \phi_j$.
\end{lemma}

\begin{proof}
If $E^*_{\mathcal P,\Phi}(g,s) = E^*_{\mathcal P,\Phi}(g,\mu(s))$, then it must be true that for every $g$ the product of the divisor sum and the Whittaker function appearing in $FW_{\mathcal P,\Phi}(g,(p,1,\ldots,1),s)$ is invariant under $s\mapsto \mu(s)$.  Since we can choose $g=\diag(p^{-1},1,\ldots,1)$, for which the Whittaker function portion of $FW_{\mathcal P,\Phi}$ is independent of $p$, it must be the case that the divisor sum by itself is invariant.  Hence, it suffices to prove that the only possible affine transformations $\mu$ which preserve the divisor sum are in fact permutations $\sigma$ for which $\sigma\Phi=\Phi$.
\end{proof}

\vskip 4pt

\noindent
{\bf Proof of Proposition~\ref{prop:OneBlock}}. 
In this case, note that \eqref{eq:DivisorSum} simplifies to give
 \[ \lambda(p)\big(p^{s_1} + p^{s_2}+ \cdots + p^{s_r}\big) = \lambda(p)\big( p^{\mu_1(s)} + p^{\mu_2(s)} + \cdots p^{\mu_r(s)}\big). \]
We may assume that $p$ is a prime for which $\lambda(p)\neq 0$, hence
\begin{equation}\label{eq:noLambda}
 p^{s_1} + p^{s_2}+ \cdots + p^{s_r} = p^{\mu_1(s)} + p^{\mu_2(s)} + \cdots p^{\mu_r(s)}.
\end{equation}
By Lemma~\ref{lem:UniqueFEiffDivSum}, we just need to show that the only way this can possibly hold is if each term $\mu_j(s)$ is actually equal to $s_{\sigma(j)}$ for some permutation $\sigma\in S_r$.

To see that this is the case, first fix $s_2,\ldots,s_{r-1}$ and assume that $s_1\in \R$ with $s_1\to \infty$.  Then in order for the asymptotics of the left hand side of \eqref{eq:noLambda} to agree with those of the right hand side, it must be the case that $\mu_{j_1}(s) = s_1$ for some $j_1\in \{1,2,\ldots,r\}$.  This same argument gives, for each $i=1,2,\ldots,r-1$, that $s_i = \mu_{j_i}(s)$ for some $j_i\in \{1,2,\ldots,r\}$.  Similarly, we see that $\mu_{j_r}(s)=s_r$ by looking at the case that $s_1\to-\infty$ with $s_2,\ldots,s_{r-1}$ fixed.  Therefore, $\mu$ is given by the map $i\mapsto j_i$, which, by the pigeonhole principle, is a permutation. \hfill \qedsymbol

\vskip 6pt

\noindent
{\bf Proof of Proposition~\ref{prop:IfMuIsLinear}}.
We assume that $\Phi$ is as in Lemma~\ref{lem:UniqueFEiffDivSum} and that $\mu$ is linear, i.e., \eqref{eq:UniqueFE} holds with $b_1=b_2=\cdots=b_r=0$.

In order to simplify the proof, we set some notation.  Recall, first, that $\hat{r}_{i} = \sum\limits_{j=1}^i r_j$.  Then let
 \[ I_j := \big\{ i \in \Z \mid \hat{r}_{j-1} < i \leq \hat{r}_j \big\}. \]
For each $j=1,2,\ldots,N$, we set $\lambda_j:=\lambda_{\phi_j}$.  Then \eqref{eq:UniqueFEiffDivSum} is equivalent to
\begin{equation}\label{eq:DivisorSum2}
 \sum_{j=1}^N \lambda_j(p)\Bigg( \sum_{i\in I_j} p^{s_i}\Bigg) =  \sum_{j=1}^N \lambda_j(p)\Bigg( \sum_{i\in I_j} p^{\mu_i(s)}\Bigg).
\end{equation}

As in the proof of Proposition~\ref{prop:OneBlock}, we choose $p$ such that $\lambda_j(p)=\lambda_k(p)$ if and only if $j=k$.  Then, for a particular $i\in I_1$, consider the limit $s_i\to+\infty$ (where $s_j$ is fixed for $j\neq i,r$).  Then the left hand side of \eqref{eq:DivisorSum2} is asymptotic to $\lambda_1(p)p^{s_i}$.  Since we are assuming Conjecture~\ref{distinct}, to agree with the right hand side, it must be the case that $\mu_{j_i}(s)=s_i$ for some choice of $j_i\in I_1$.  This shows, again via the pigeonhole principle, that $\mu$ permutes the variables $\big\{s_1,\ldots, s_{r_1}\big\}$.

The same argument holds for $i\in I_k$ for $k=2,\ldots,r$ by considering $s_i\to +\infty$ with $s_j$ fixed for $j\neq i,1$.  Comparing the asymptotics of both sides of \eqref{eq:DivisorSum2}, we conclude that $\mu$ permutes the set $\big\{ s_i \mid i\in I_k\big\}$.  Combined with Lemma~\ref{lem:UniqueFEiffDivSum}, this completes the proof. \hfill \qedsymbol

\begin{remark}\rm
We observe that if one is in a case that $\mu$ is linear, i.e., that $\mu$ is as in \eqref{eq:UniqueFE} with the constants $b_i=0$ for each $i=1,2,\ldots,r$, then the representation of $\mu$ as a matrix is not unique.  Indeed, we can think of $\mu$ as an element in the image of the natural map $\psi: M_{r\times r}(\C) \to \Hom_\C( V,\C^r )$, where
 \[ V=\C^r/\{ (s_1,s_2,\ldots,s_r) \mid s_1+s_2+\cdots+s_r=0 \}. \]
Note that $\psi$ is clearly surjective, but its kernel contains the subspace 
 \[ W = \left \{ \left. \left[ \begin{smallmatrix} a_1 & a_1 & \cdots & a_1 \\ a_2 & a_2 & \cdots & a_2 \\ \vdots & \vdots & & \vdots \\ a_r & a_r & \cdots & a_r \end{smallmatrix}\right] \right| a_1,a_2,\ldots,a_r \in \C \right\}. \]
By a simple dimension counting argument, in fact, we see that $\ker(\psi)=W$.
\end{remark}

%%%%%%%%%%%%%%%%%%%%%%%%%%%%%%%%%%%%%%%%%%%%%%%%%%%%%%%%%%%%%%%%%%%%%%

%%%%%%%%%%%%%%%%%%%%%%%%%%%%%%%%%%%%%%%%%%%%%%%%%%%%%%%
%%% Acknowledgements. ÖÂÐ»
%%%%%%%%%%%%%%%%%%%%%%%%%%%%%%%%%%%%%%%%%%%%%%%%%%%%%%%

%%%%%%%%%%%%%%%%%%%%%%%%%%%%%%%%%%%%%%%%%%%%%%%%%%%%%%%
%%% Conflict of interest. ×÷ÕßÀûÒæÉùÃ÷
%%%%%%%%%%%%%%%%%%%%%%%%%%%%%%%%%%%%%%%%%%%%%%%%%%%%%%%
%\InterestConflict

%%%%%%%%%%%%%%%%%%%%%%%%%%%%%%%%%%%%%%%%%%%%%%%%%%%%%%%
%%% Supplements. ²¹³ä²ÄÁÏ, ·Ç±ØÑ¡
%%%%%%%%%%%%%%%%%%%%%%%%%%%%%%%%%%%%%%%%%%%%%%%%%%%%%%%
%\Supplements{}

%%%%%%%%%%%%%%%%%%%%%%%%%%%%%%%%%%%%%%%%%%%%%%%%%%%%%%%
%%% Reference section. ²Î¿ŒÎÄÏ×
%%% citation in the content using "some words~\cite{1,2}".
%%% ~ is needed to make the reference number is on the same line with the word before it.
%%%%%%%%%%%%%%%%%%%%%%%%%%%%%%%%%%%%%%%%%%%%%%%%%%%%%%%

\end{document}